\documentclass[a4paper,11pt,twoside]{article}

\usepackage{amssymb,amsmath,amsthm}
\usepackage[utf8]{inputenc}
\usepackage{fancyhdr}
\usepackage{layout}
\usepackage{amsopn}
\usepackage[letterpaper, left=1in, right=1in]{geometry}
\usepackage{amstext}
\usepackage{hyperref}
\usepackage{lastpage}
\usepackage{alltt}
\usepackage{mathrsfs, MnSymbol}
\usepackage{enumerate}
\usepackage{setspace}
\usepackage[shortlabels]{enumitem}

\theoremstyle{plain}

\newtheorem{theorem}{Theorem}
\newtheorem*{theorem*}{Theorem}
\newtheorem{lemma}[theorem]{Lemma}
\newtheorem*{lemma*}{Lemma}
\newtheorem{prop}[theorem]{Proposition}
\newtheorem{corollary}[theorem]{Corollary}
\newtheorem{fact}[theorem]{Fact}
\newtheorem*{fact*}{Fact}

\theoremstyle{definition}

\newtheorem*{definition*}{Definition}
\newtheorem*{quest*}{Question}

\setlist[enumerate]{noitemsep, topsep=1pt}
\setlist[itemize]{noitemsep,topsep=.5pt}

\newcommand{\Z}{\mathbb{Z}}
\newcommand{\N}{\mathbb{N}}
\newcommand{\chr}[2]{\ensuremath{{#1} \,\text{char}~ {#2}}}

\newcommand{\plnorm}{\triangleleft}
\newcommand{\order}[1]{\ensuremath{{\lvert {#1} \rvert}}}
\newcommand{\dv}{\ensuremath{\,\Big\lvert\,}}
\newcommand{\cyclic}[1]{\ensuremath{\langle {#1} \rangle}}

\newcommand{\Syl}[2]{\ensuremath{\text{Syl}_{{#1}}({#2})}}

\begin{document}

\title{\textbf{A Generalization of the Hughes Subgroup}}
\author{Mark L. Lewis and Mario Sracic}
\date{}
\maketitle

\begin{abstract}
Let $G$ be a finite group, $\pi$ be a set of primes, and define $H_{\pi}(G)$ to be the subgroup generated by all elements of $G$ which do not have prime order for every prime in $\pi$. In this paper, we investigate some basic properties of $H_{\pi}(G)$ and its relationship to the Hughes subgroup. We show that for most groups, only one of three possibilities occur: $H_{\pi}(G) = 1$, $H_{\pi}(G)=G$, or $H_{\pi}(G) = H_{p}(G)$ for some prime $p \in \pi$.  There is only one other possibility: $G$ is a Frobenius group whose Frobenius complement has prime order $p$, and whose Frobenius kernel, $F$, is a nonabelian $q$-group such that $H_{\pi}(G)$ arises as the proper and nontrivial Hughes subgroup of $F$.  We investigate a few restrictions on the possible choices of the primes $p$ and $q$.
\end{abstract}

\textbf{Mathematics Subject Classification.}\quad 20D25. \\

\textbf{Keywords.}\quad Hughes subgroup, Frobenius groups, elements of prime order. \\

\begin{enumerate}
\item[\textbf{1.}] \textbf{Introduction}
\end{enumerate}

In 1957, D. R. Hughes posed the following problem: Let $G$ be any group and $p$ be a prime.  Consider the following subgroup,
	\[
	H_{p}(G) := \cyclic{x \in G : x^{p} \neq 1},
	\]
or $H_{p}$ when $G$ is clear from the context, which we call the Hughes subgroup of $G$ relative to $p$.  Hughes asked, ``is the following conjecture true: either $H_{p} =1, H_{p}=G$, or $\order{G:H_{p}} =p$?'' \cite{Hughes58}.  Hughes had proved this conjecture for $p=2$ two years prior \cite{Hughes56}, and shortly thereafter, Straus and Szekeres \cite{StrSzek58} answered in the affirmative for $p=3$.  The conjecture was settled completely for finite non-$p$-groups by Thompson and Hughes in 1959 \cite{HughThomp59}.  Furthermore, Hughes and Thompson defined and classified $H_{p}$-groups: a nontrivial finite group $H$ is an $H_{p}$-group if it arises as $H_{p}(G)$, for some finite group $G$, with index $p$.  The conjecture was shown to be false (in general) for $p$-groups ($p \geq 5$) by G. E. Wall \cite{Wall67} through the construction of a counterexample, $G$, with $\order{G:H_{5}}=25$.

The study of Hughes subgroups has proven to be a rich area of study given the natural extension of Frobenius groups.  However, our focus is a natural generalization of the subgroup itself similar to that as described in \cite{ErcanGul90}:  Let $G$ be a finite group and $n \in \N$.  The generalized Hughes subgroup of $G$, relative to $n$, is defined as
	\[
	H_{n}(G) := \cyclic{x \in G : x^{n} \neq 1}.
	\]
We introduce a new generalization as follows: Let $\pi$ be a set of primes and consider the following subgroup,
	\[
		H_{\pi}(G) := \cyclic{x \in G : x^{p}\neq 1~\text{for all}~p \in \pi}.
	\]
In this paper we examine some basic properties of $H_{\pi}(G)$, the influence of $H_{\pi}(G)$ on the structure of a finite group $G$, and its relationship with the Hughes subgroup.  We will see that for most groups, only one of three possibilites occur: (1) $H_{\pi}(G)=G$, (2) $H_{\pi}(G) = H_{p}(G)$ for some prime $p \in \pi$, or (3) $H_{\pi}(G) = 1$.  It will be shown that only one other possibility can occur and we can characterize this possibility.  In particular, we prove the following:

\begin{theorem}\label{mainthm}
Let $G$ be a finite solvable group and set $\pi=\pi(G)$.  Then $1 < H_{\pi}(G) < H_{p}(G)$ for all $p \in \pi$ if and only if $G$ is a Frobenius group whose Frobenius kernel, $F$, is a nonabelian $q$-group such that $1 < H_{q}(F) < F$, and whose Frobenius complement has prime order.  In this case, $H_{\pi}(G) = H_{q}(F)$.
\end{theorem}

At this point, it remains an open question whether such groups as in Theorem \ref{mainthm} actually exist.  The existing body of work on the (original) Hughes subgroup provides a means to eliminate a few possible choices for the primes $p$ and $q$.  In particular, $p \geq 7$ and $q \geq 5$.

We would like to thank Professor Khukhro for several helpful comments while writing this paper.

\begin{enumerate}
\item[\textbf{2.}] \textbf{Preliminaries}
\end{enumerate}
	
One can always find examples of proper and nontrivial generalized Hughes subgroups by considering Frobenius actions.  For example, the natural action of the multiplicative group of a field on the additive group of said field.  For a given finite group $G$, we can restrict the sets of primes to consider in computing $H_{\pi}(G)$.  Clearly, if every element of $G$ has prime order or $\pi \cap \pi(G) \neq \emptyset$, then $H_{\pi}$ is either trivial or improper, respectively. 

\begin{fact}
Let $G$ be a finite group.  If $\pi_{1}$ and $\pi_{2}$ are two sets of primes such that $\pi_{1} \subseteq \pi_{2}$, then $H_{\pi_{2}}(G) \leqslant H_{\pi_{1}}(G)$.
\end{fact}

\begin{fact}
Let $G$ be a finite group, $\pi$ be a set of primes, and set $\pi_{0}=\pi\cap \pi(G)$.  Then $H_{\pi}(G) = H_{\pi_{0}}(G)$.
\end{fact}

Consequently, we may assume $\pi \subseteq \pi(G)$, and further, $\order{\pi}\geq 2$; otherwise $H_{\pi}(G)$ coincides with the Hughes subgroup relative to the unique prime in $\pi$.  From the definitions of $H_{\pi}$ and $H_{p}$, we obtain the following set of inclusions:
	\begin{equation}\label{eqn1}
	1 \leqslant H_{\pi}(G) \leqslant \bigcap_{p \in \pi} H_{p}(G) \leqslant G,
	\end{equation}
which lead to questions regarding what conditions, if any, ensure that the inclusions in (\ref{eqn1}) are proper.  To this end, we first address intersections of Hughes subgroups.  

\begin{prop}\label{prop1}
Let $G$ be a finite group with $\order{\pi(G)}\geq 2$.  Then
	\begin{eqnarray}
	\bigcap_{p\in\pi(G)} H_{p}(G) = \left\{\begin{array}{lr}
	H_{p}(G), & \text{for a unique prime $p\in\pi(G)$,} \\
	G, & \text{otherwise}. \\
	\end{array}\right.
	\end{eqnarray}
\end{prop}

\begin{proof}
	If $H_{p}(G)=G$ for all $p \in \pi(G)$, then there is nothing to show.  Suppose there are distinct primes $p,q \in \pi(G)$ such that $H_{p} < G$ and $H_{q}< G$, and note both Hughes subgroups are necessarily nontrivial.  By a simple set theoretic argument, we have $(G\setminus H_{p}) \cap (G\setminus H_{q}) = G \setminus (H_{p}\cup H_{q})$.  Inasmuch as the only element $x$ to simultaneously satisfy $x^{q} = 1= x^{p}$ is the identity, the left-hand-side is empty.  Thus $G = H_{p} \cup H_{q}$ which is impossible.
\end{proof}

We note that explicit examples exist where $H_{\pi}(G)$ is properly contained in a Hughes subgroup, albeit at the cost of being trivial.  The simplest example is the Frobenius group $G=S_{3}$, where
\[
1=H_{\{2,3\}}(S_{3}) < H_{2}(S_{3}) < S_{3}.
\]
A second example is to take $E = GF(3^{3})$ and consider the natural Frobenius action of the subgroup $H$ of $E^{\times}$ of order 13 on the additive group $N$ of $E$.  The Galois group, $G=Gal(E/\Z_{3})$ acts naturally on the resulting Frobenius group $\Gamma_{0} = NH$, and so we can consider the semidirect product $\Gamma = \Gamma_{0}G$.  In this example, we have 
\[
1=H_{\{3,13\}}(\Gamma) < H_{13}(\Gamma_{0})=N < H_{3}(\Gamma) = \Gamma_{0} < \Gamma.
\]

\begin{enumerate}
\item[\textbf{3.}] \textbf{Results and Proof of Theorem}
\end{enumerate}

Next, we consider what influence $H_{\pi}(G)$ has on the structure of a finite group $G$.  By definition, every element of $G \setminus H_{\pi}(G)$ has prime order for some prime in $\pi$.  Thus, if $H_{\pi}(G)$ is trivial, then all nonidentity elements of $G$ have prime order, and such groups were completely classified by Deaconescu \cite{Deaconescu89}, and Cheng, et. al. \cite{Cheng93}.  Otherwise, we fall under a situation investigated by Qian \cite{Qian2005}: 

\begin{theorem}[``Theorem 1'', \cite{Qian2005}]
Let $G$ be a finite group and $N\plnorm G$ such that every element of $G \setminus N$ is of prime order.  Then $G$ has one of the following structures:
	\begin{enumerate}[i$)$.]
	\item $G = A_{5}$ and $N=1$.
	\item $G = F \rtimes A$ is a Frobenius group, where the complement, $A$, is of prime order, and the Frobenius kernel, $F$, is of prime power order, with $N < F$.
	\item $G$ is a $p$-group.
	\item $G = (L \times K) \rtimes A$, where $L \times K$ is nilpotent, $A$ is of prime order $p$, $K \rtimes A \in \Syl{p}{G}$, $L \plnorm G$ is a Hall $p'$-subgroup of $G$, $A$ acts fixed-point-freely on $L$, and $N = L \times K$.
	\end{enumerate}
\end{theorem}

In the pursuit of cases where $H_{\pi}$ is nontrivial, we may assume $G$ is solvable; otherwise, $G = A_{5}$ by the above.  Towards this end, we establish a minor relationship between $H_{\pi}(G)$ and Hughes subgroups.

\begin{lemma}\label{lem1}
Let $G$ be a finite solvable group, $\pi$ be a nonempty subset of $\pi(G)$, and assume $\order{\pi(G)} \geq 2$.  Then $H_{\pi}(G) = G$ if and only if $H_{p}(G) = G$ for all $p \in \pi$.
\end{lemma}

\begin{proof}
Observe that one direction is a triviality as $H_{\pi}(G) \leqslant \bigcap_{p \in \pi} H_{p}(G)$.

Suppose $H_{\pi}(G) <  G$ and let $M$ be a maximal normal subgroup of $G$ containing $H_{\pi}(G)$.  Then $\order{G:M}=p$ for some prime $p \in \pi(G)$ and since $\exp(G/M) \mid \prod_{q\in \pi} q$, we have $p \in \pi$.  By hypothesis, there exists a generator $x$ of $H_{p}(G)$ such that $x \notin M$.  Hence $x^{q}=1$ for some $q \in \pi\setminus\{p\}$.  Consequently, $Mx$ is a $q$-element of the $p$-group $G/M$ implying $x \in M$, a contradiction.  Therefore, $H_{\pi}(G) = G$.
\end{proof}

We are now ready to prove the main theorem, which we restate below.

\begin{theorem*}
Let $G$ be a finite solvable group and set $\pi=\pi(G)$.  Then $1 < H_{\pi}(G) < H_{p}(G)$ for all $p \in \pi$ if and only if $G$ is a Frobenius group whose Frobenius kernel, $F$, is a nonabelian $q$-group such that $1 < H_{q}(F) < F$, and whose Frobenius complement has prime order.  In this case, $H_{\pi}(G) = H_{q}(F)$.
\end{theorem*}

\begin{proof}
Suppose $1 < H_{\pi}(G) < H_{p}(G)$ for all $p \in \pi$.  By Lemma \ref{lem1} and Proposition \ref{prop1}, there exists a unique prime $p \in \pi$ such that $H_{p}(G) < G$.  It follows from (\cite{Qian2005}, Theorem 1) that $G = FA$ is a Frobenius group with complement, $A$, of prime order $p$, and kernel, $F$, a $q$-group.  Now $F = H_{p}(G)$ by (\cite{HughThomp59}, Theorem 2) and so $H_{\pi}(G) < F$.  In particular, we have $H_{\pi}(G) = H_{q}(F)$.  If there exists $z \in Z(F) \setminus H_{\pi}(G)$, then $1=(zh)^{q} = h^{q}$ for all $h \in H_{\pi}(G)$ contrary to our assumption.  Therefore, $Z(F) \leqslant H_{\pi}(G)$ (in fact, $Z(F) < H_{\pi}(G)$ by similar reasoning) and $F$ satisfies the conclusion.

The converse follows immediately from the hypotheses.
\end{proof}

The statement of Theorem \ref{mainthm} relied upon taking our set of primes to be precisely $\pi(G)$.  The following Corollary addresses the natural question of what can be said of $H_{\pi}(G)$ when $\pi \subset \pi(G)$.

\begin{corollary}
Let $G$ be a finite solvable group, $\order{\pi(G)}\geq 2$, $\pi \subset \pi(G)$, and suppose there exists $p\in\pi(G)$ such that $H_{p}(G) < G$.
	\begin{enumerate}[i$)$.]
	\item If $p \notin \pi$, then $H_{\pi}(G)=G$.
	\item If $p \in \pi$, then $H_{\pi}(G)=H_{p}(G)$.
	\end{enumerate}
\end{corollary}

\begin{proof}
	First, the results are trivial if $\order{\pi(G)}=2$ and so we may assume $\order{\pi(G)}>2$.  Now ($i$) follows immediately from Proposition \ref{prop1} and Lemma \ref{lem1}, whereas ($ii$) follows from Theorem \ref{mainthm} given our assumption.
\end{proof}

\begin{enumerate}
\item[\textbf{4.}] \textbf{Further Exploration}
\end{enumerate}

As indicated by Theorem \ref{mainthm}, groups with proper and nontrivial $H_{\pi}$-subgroups have a very restrictive structure, and explicit examples remain elusive.  To perhaps stimulate some additional investigation into this specific case, we pose the following question:

\begin{quest*}\label{questA}
Does there exist a Frobenius group, $FA$, where the kernel, $F$, is a nonabelian $q$-group which satisfies $1 < H_{q}(F) < F$, and the complement, $A$, is of prime order $p$? 
\end{quest*}

In the pursuit of an example, there are a few restrictions that can be made about the primes $p$ and $q$.  First, as $F$ is nonabelian, the Frobenius complement $A$ cannot be even (\cite{Isaacs}, Theorem 6.3); in particular, $A$ cannot be a 2-group, ie. $p \neq 2$.  Next, we know the Hughes conjecture holds for the primes 2 and 3, so if $q=2,3$, then $\order{F:H_{q}(F)}=q$.  Since $\chr{H_{q}(F)}{F}$, $F/H_{q}(F)$ admits a Frobenius action by $A$ implying $p=\order{A} \dv q-1$.  However, as $A$ is nontrivial and not a 2-group, we must have $q \neq 2,3$.

We would like to thank Evgeny Khukhro for inspiring the following by a helpful remark highlighting the relationship between regular $p$-groups and the Hughes subgroup.  Recall that a $p$-group $G$ is called regular if, for all $x,y \in G$, $x^{p}y^{p}=\prod_{i} z_{i}^{p}$ where $z_{i} \in \cyclic{x,y}'$, and we shall rely on a few results from Section III.10 of \cite{Huppert1967} to help prove the following:

\begin{lemma}\label{lem2}
Let $G$ be a regular $p$-group.  Then $H_{p}(G) = 1$ or $G$.
\end{lemma}

\begin{proof}
By (\cite{Huppert1967}, III.10.5 Hauptsatz), $\Omega_{1}(G) = \{g \in G : g^{p}=1\}$ (note it is the set of such elements). Clearly $G \setminus \Omega_{1}(G) \subseteq H_{p}(G)$ and it follows that $G = H_{p}(G) \cup \Omega_{1}(G)$.  Since $G$ cannot be the union of two proper subgroups, either $H_{p}(G) = G$ or $\Omega_{1}(G) = G$ forcing $H_{p}(G)=1$.
\end{proof}

The following corollary is clear from (\cite{Huppert1967}, III.10.2 Satz) and Lemma \ref{lem2}.

\begin{corollary}\label{cor1}
If $G$ is a $p$-group with $1 < H_{p}(G)<G$, then $G$ has nilpotence class, $c(G)$, at least $p$.
\end{corollary}

Returning to the situation posed in our Question, we have from Corollary \ref{cor1} that $F$ must have nilpotence class at least $q$, where $q>3$.

\begin{lemma}\label{lem3}
Let $G$ be a $q$-group with $1<H_{q}(G)<G$ and admitting a fixed-point-free (f.p.f) automorphism $\phi$ of prime order $p$. Then $G$ cannot be metabelian.
\end{lemma}

\begin{proof}
Suppose $G$ is metabelian. Since metabelian groups satisfy the Hughes conjecture (\cite{HoganKappe}, Theorem 1), we have $\order{G:H_{q}(G)}=q$.  It follows that $\phi$ induces a f.p.f. automorphism of order $p$ on $G/H_{q}(G)$.  In particular, $G/H_{q}(G) \rtimes \cyclic{\phi}$ is a Frobenius group and so $p = \order{\phi} \dv \order{G:H_{q}(G)}-1 = q-1$.  Now a result of V.A. Kreknin and A.I. Kostrikin (\cite{Kreknin}, Theorem B) gives a bound on the nilpotence class of a group in terms of its derived length and the order of a fixed-point-free automorphism of the group.  In particular, as the derived length of $G$ is 2, it follows from (\cite{Kreknin}, Theorem B):
	\[
		c(G) < \frac{(p-1)^{2}-1}{p-2} = \frac{p^{2}-2p}{p-2}=p \leq q-1 < q.
	\]
However, this contradicts Corollary \ref{cor1} and so $G$ cannot be metabelian.
\end{proof}

By the above result, we see that $F$ cannot be metabelian.  From this we can place additional constraints on the order of the Frobenius complement. If $p=3$, then as $F$ is finite and admits a fixed-point-automorphism of order 3, we have by (\cite{Neumann56}, `Theorem') that $F$ has class at most 2, a contradiction. If $p=5$, then by G. Higman's work in (\cite{Higman}, Section 5), we have $c(F)\leq 6$. Yet $c(F) \geq q$, we have already shown $q \neq 2,3,$ and we know $q \neq p$, this yields another contradiction.  Therefore, $p \neq 2, 3$, or $5$.

\begin{corollary}
In the situation of our Question, $q$ cannot be any prime such that $\pi(\order{F}-1) \subseteq \{2,3,5\}$.
\end{corollary}

\bibliographystyle{plain}
\bibliography{PreliminaryGenHughesSubgroup}

\vspace{5mm}

\begin{minipage}[t]{0.5\textwidth}
\noindent MARK L. LEWIS \\
Department of Mathematical Sciences, \\
Kent State University, \\
Kent, OH 44242, \\
USA \\
e-mail: \texttt{lewis@math.kent.edu} \newline 
\end{minipage}
\begin{minipage}[t]{0.5\textwidth}
\noindent MARIO SRACIC \\
Department of Mathematical Sciences, \\
Kent State University, \\
Kent, OH 44242, \\
USA \\
e-mail: \texttt{msracic@kent.edu}
\end{minipage}

\end{document}